\documentclass[12pt]{extarticle}
\RequirePackage{etex}
\usepackage{enumerate}
\usepackage{multicol}
\usepackage{etex}
\usepackage[utf8]{inputenc}
\usepackage{amsfonts}
\usepackage{xcolor}
\usepackage{comment}
\usepackage{amsmath}
\usepackage{amssymb}
\usepackage{amsthm} 
\usepackage{mathrsfs}
\usepackage{stmaryrd}
\usepackage{color}
\usepackage[english]{babel}
\usepackage{fontenc}
\usepackage{url}
\usepackage{graphicx}
\usepackage{diagbox}
\usepackage{enumitem}
\usepackage{bm}       
\usepackage{hyperref}
\hypersetup{
    colorlinks=true,
    linkcolor=blue!80!black,,
    citecolor=blue!50!black,,
    urlcolor=blue
}
\usepackage{caption}
\usepackage{epstopdf}
\usepackage{hyphenat}
\usepackage{float}

\usepackage{tikz-cd}
\usetikzlibrary{shapes,arrows,fit,automata}
\usetikzlibrary{positioning, shapes}
\usetikzlibrary{decorations.pathreplacing}
\usepackage{algorithm}
\usepackage[noend]{algpseudocode}
\usepackage[margin=1in]{geometry} 

\makeatletter
\def\algbackskip{\hskip-\ALG@thistlm}
\makeatother

\usepackage[export]{adjustbox}
\usepackage{tikz}
\usepackage{tikz-qtree}
\usepackage{color}
\usepackage[all]{xy}
\usetikzlibrary{matrix,arrows,decorations.pathmorphing}
\usepackage{booktabs}
\usepackage{array}
\usepackage{nicematrix}
\usepackage{graphicx}
\usepackage{adjustbox}
\usepackage{soul}

\numberwithin{equation}{section}
\newtheorem{theorem}{Theorem}[section]

\newtheorem{lemma}[theorem]{Lemma}
\newtheorem{corollary}[theorem]{Corollary}
\newtheorem{proposition}[theorem]{Proposition}

\theoremstyle{definition}

\newtheorem{definition}[theorem]{Definition}

\newtheorem{examplex}[theorem]{Example}
\newenvironment{example}
{\pushQED{\qed}\examplex}
{\popQED\endexamplex}

\newtheorem{remarkx}[theorem]{Remark}
\newenvironment{remark}
{\pushQED{\qed}\remarkx}
{\popQED\endremarkx}

\newtheoremstyle{citing}
{}
{}
{\itshape}
{}
{\bfseries}
{\textbf{.}}
{.5em}
{\thmnote{#3}}
{\theoremstyle{citing}
}

\DeclareMathOperator{\rank}{rank}

\newcommand{\PP}{\mathbb{P}}

\newcommand{\cV}{\mathcal{V}}
\newcommand{\cM}{\mathcal{M}}

\newcommand{\cS}{\mathcal{S}}

\newcommand{\cX}{\mathcal{X}}

\newcommand{\NN}{\mathbb{N}}
\newcommand{\CC}{\mathbb{C}}
\newcommand{\RR}{\mathbb{R}}
\newcommand{\ZZ}{\mathbb{Z}}

\title{Algebraic Complexity and Neurovariety of Linear Convolutional Networks}
\author{Vahid Shahverdi}

\date{}

\begin{document}

\maketitle 
\begin{abstract}
In this paper, we study linear convolutional networks with one-dimensional filters and arbitrary strides. The neuromanifold of such a network is a semialgebraic set, represented by a space of polynomials admitting specific factorizations. Introducing a recursive algorithm, we generate polynomial equations whose common zero locus corresponds to the Zariski closure of the corresponding neuromanifold. Furthermore, we explore the algebraic complexity of training these networks employing tools from metric algebraic geometry. Our findings reveal that the number of all complex critical points in the optimization of such a network is equal to the generic Euclidean distance degree of a Segre variety. Notably, this count significantly surpasses the number of critical points encountered in the training of a fully connected linear network with the same number of parameters.
\end{abstract}
\section{Introduction}
 A neural network can be viewed as a family of functions that are parameterized based on its architecture. These functions collectively form a finite-dimensional manifold in the space of continuous functions, often referred to as the \emph{neuromanifold} or \emph{function space} of the network. The goal of training a network is to find the optimal parameters for which its corresponding function captures the underlying pattern of the given data by minimizing a loss function. This optimization task has been widely studied in the realm of machine learning e.g., \cite{pmlr-v119-dukler20a, pmlr-v97-allen-zhu19a, du2018gradient, NEURIPS2018_5a4be1fa}. 

This paper focuses on a specific category of neural networks known as \emph{linear networks}, which employ linear activation functions. Despite relying exclusively on linear functions, these networks are considered as simplified models for the analysis of more complex neural networks. Linear networks and their training process have been studied in \cite{NIPS1988_123, 10.1109/72.392248, NIPS2016_6112, zhou2018critical, pmlr-v80-laurent18a, geometryLinearNets}.

Examples of linear networks include \emph{one-dimensional linear convolutional networks} (1D-LCNs). These networks consist of convolutions with one-dimensional filters. Their architectures are defined by a tuple of filter sizes and strides. Noteworthy studies closely related to our current work include \cite{LCN} and \cite{kohn2023function}, both of which explore the geometry of the neuromanifolds and optimization tasks in 1D-LCNs. These studies reveal that the neuromanifolds of 1D-LCNs are semialgebraic sets consisting of polynomials with specific factorizations. They present various theoretical findings on the neuromanifolds including their dimensions, relative boundaries, and singular points. Furthermore, these studies provide insights into the position of critical points involved in optimizing 1D-LCNs.

In this study, we explore other key aspects of 1D-LCNs, aiming to answer two fundamental questions:

\begin{enumerate}
    \item \textbf{Algebraic Structure:} 
    Can we find polynomial equations whose common zero locus is the Zariski closure of a given 1D-LCN neuromanifold?
    \item \textbf{Algebraic Complexity:} How many critical points exist in the optimization of a 1D-LCN using the quadratic loss function?
\end{enumerate}

 To address the first question, we present an algorithm for generating polynomial equations whose common zero locus corresponds to the Zariski closure of the neuromanifold. In our method, we translate the sparsity of polynomials in a 1D-LCN into common factoring conditions. This condition is studied thoroughly using the classical theory of resultants  \cite{kakie1976resultant}. 
  We compare our algorithm with a conventional approach in \texttt{Macaulay2} \cite{m2}. Our algorithm has a notable speed advantage, surpassing the conventional approach in terms of computational time. However, our method does not yield a radical ideal.

 For the second question, we use the tools from Euclidean distance optimization. The \emph{(generic) Euclidean distance degree}
 is an invariant that quantifies the algebraic complexity of the closest point problems within an algebraic variety. This concept was first introduced in \cite{draisma2016euclidean}, and it is studied in great detail in \cite{MetricAlgGeo}. We see that optimizing a linear network using the quadratic loss function entails finding the closest point using the weighted Frobenius norm on the neuromanifold. For instance, considering the space of bounded rank matrices as the neuromanifold of a fully connected linear network, the number of critical points arising while training such a network is determined by the Eckart--Young Theorem \cite{kohn2024geometry}. In the case of a 1D-LCN, we demonstrate that the neuromanifold is derived through a Segre map followed by a linear morphism. Furthermore, we demonstrate that the count of all complex critical points involved in the training process of a 1D-LCN is determined by the generic Euclidean distance degree of a Segre variety. This reveals that the algebraic complexity of training a 1D-LCN is significantly higher than that of its fully connected counterpart with the same number of parameters. 
 
The structure of our paper is as follows: In Section \ref{sec:pre}, we revisit the definitions and fundamental properties of 1D-LCNs and their neuromanifolds. Additionally, we explore the domain of Euclidean distance optimization over an algebraic variety. Moving forward, Section \ref{subsec:ideal_functionspace} is dedicated to introducing our algorithm, a key player in generating an ideal where the common zero locus seamlessly aligns with the Zariski closure of our neuromanifold. Finally, our exploration in Section \ref{sec:Eucl} navigates the details of the algebraic complexity entailed in training a 1D-LCN with the quadratic loss function. For convenience, in Table \ref{tab:symbols}, we list essential notations and their descriptions related to 1D-LCNs.

\section{Preliminaries}
\label{sec:pre}
In this section, we delve into the properties of one-dimensional linear convolutional networks and their neuromanifolds within the scope of our discussion. Furthermore, to investigate the algebraic complexity of our networks, we recall various notions of Euclidean distance degrees. These key tools are integral to the subsequent sections. 
\subsection{Convolutions and Convolutional Linear Networks}
A one-dimensional linear convolutional network (1D-LCN)  is a type of neural network designed for processing one-dimensional sequences or signals. The architecture of such a network is determined by a tuple of filter sizes $\mathbf{k}=(k_1,\ldots,k_L)\in \NN^{L}$, strides $\mathbf{s}=(s_1,\ldots,s_L)\in \NN^L$, and the input dimension
$d_0\in \ZZ^{>0}$. A 1D-LCN contains linear maps $\alpha:\RR^{d_0} \to \RR^{d_L}$ such that $\alpha$ can be written as 
\begin{equation}
    \label{eq:comp-convo}
    \alpha_{w_L,s_L}\circ \cdots \circ \alpha_{w_1,s_1},
\end{equation}
 where $w_\ell \in \RR^{k_\ell}$ is a filter and each $\alpha_{w_\ell,s_\ell}$ is a \emph{convolution}, defined by
\begin{align}
\label{eq:convoo}
\alpha_{w_\ell,s_\ell}:\RR^{d_{\ell-1}} \to \RR^{d_\ell}, \quad {x} \, \mapsto \, \left [ \sum_{j=0}^{k-1} w[j] \cdot x[is+j] \right ]_i^\top,
\end{align}
with $d_\ell = \frac{d_{\ell-1}-k_\ell}{s_\ell}+1$ for $\ell=1,\ldots,L$. 

The composition of $L$ convolutions in \eqref{eq:comp-convo} is again a convolution $\alpha_{w,s}:\RR^{d_0} \to \RR^{d_L}$ with filter size $k:=k_1+\sum_{l=2}^L(k_l-1)\prod_{i=1}^{l-1}s_i$ and stride $s:= s_1 \cdots s_L$ (\cite[Proposition~2.2]{LCN}).
 In our analysis, we intentionally exclude considerations of input/output dimensions, unless the input/output data is our central focus. This allows us to deduce the architecture of a linear convolutional network based solely on the tuples $\mathbf{k}$ and $\mathbf{s}$.
\begin{example}
     \label{ex:L2-s_1=2}
     Consider the architecture $\mathbf{k}=(2,2)$, $\mathbf{s}=(2,1)$. 
     Then the corresponding 1D-LCN comprises convolutions $ \alpha_{w,s}= \alpha_{w_2,1}\circ \alpha_{w_1,2}$ with strides $s=2$ and filters given by
         $$w=[w_2[0]w_1[0],w_2[0]w_1[1],w_2[1]w_1[0],w_2[1]w_1[1]]^\top \in \RR^4.$$
     
\end{example}
\subsection{Neuromanifolds and Neurovarieties}
\label{sec:neuroman}
\begin{definition}
    \label{def:neuroman}
    A \emph{neuromanifold} $\cM_{\mathbf{k},\mathbf{s}} $ of a 1D-LCN architecture with filter sizes $\mathbf{k}=(k_1,\ldots,k_L)$, and strides $\mathbf{s}=(s_1,\ldots,s_L)$ is the set of all filters $w$ such that the convolution $\alpha_{w,s}$ that can be written as $\alpha_{w_L,s_L}\circ \cdots \circ \alpha_{w_1,s_1}$. In here, $\alpha_{w_\ell,s_\ell}$ is a convolution with filter $w_\ell$ of size
    $k_\ell$, and stride $s_\ell$.
\end{definition}
The neuromanifold $\mathcal{M}_{\mathbf{k},\mathbf{s}}$ is a semialgebraic set with an algebraic dimension $k_1+\cdots+k_L - (L-1)$ (\cite[Theorem~2.4]{kohn2023function}). In this paper, if $\mathcal{M}_{\mathbf{k},\mathbf{s}}$ is Zariski closed, we refer to it as \emph{neurovariety}.
Moreover, it can be parameterized by its filters, which implies $\mathcal{M}_{\mathbf{k},\mathbf{s}}$ is irreducible. For instance, in Example \ref{ex:L2-s_1=2}, the neurovariety $\cM_{(2,2),(2,1)}\subset \RR^4_{A,B,C,D}$
is a variety obtained by the zero locus of the quadratic polynomial 
$AD-BC$.

The relation between convolutions and sparse polynomials can be seen by the following isomorphism
\begin{equation} 
\label{eq:polynomials}
\pi_s(w) := \sum_{j=0}^{k-1} w[j]x^{s(k-1-j)}y^{sj},
\end{equation} 
from $\RR^k$ to the vector space  $\RR[x^s,y^s]_{k-1}$—that is, homogeneous polynomials in $x^s$ and $y^s$ of degree $k-1$. Furthermore, if $s=s_1\cdots s_L$ then the composition of convolutions $\alpha_{w_L,s_L}\circ \cdots \circ \alpha_{w_1,s_1}$ is the convolution $\alpha_{w,s}$ whose filter $w$ can be computed via polynomial multiplication
\begin{equation}
    \label{eq:pol-mult}
     \pi_1(w)=\pi_{S_L}(w_L) \cdot \pi_{S_{L-1}}(w_{L-1}) \cdots \pi_{S_3}(w_3) \cdot \pi_{S_2}(w_2) \cdot \pi_{1}(w_1),
\end{equation}
where $S_\ell := \prod_{i=1}^{\ell-1} s_i$ for $\ell>1$ (\cite[Proposition~2.2]{kohn2023function}). Note that the final stride $s_L$ has no impact on the filter $w$; therefore, from this point onward, we set $s_L=1$.

We also denote the \emph{complex neurovariety} $\mathcal{M}^{\mathbb{C}}_{\mathbf{k},\mathbf{s}}$ as the set of all complex filters that can be factorized according to the network architecture with complex filters in each layer. The following proposition describes all homogeneous polynomials associated with
its (complex) neuromanifold.
\begin{proposition}[\cite{kohn2023function}, Proposition~3.1]
\label{prop:functionspace-poly} 
If $\pi_1$ is the map defined in \eqref{eq:polynomials} and $\pi_1^\mathbb{C}$ is its complex counterpart, then:
\begin{align}
    \label{eq:line1}
    \pi_1 \left(\mathcal M_{\mathbf k,\mathbf s}\right) & = \{P\in \RR[x,y]_{k-1}: P=P_L\cdots P_1, \quad P_i\in \RR[x^{S_i},y^{S_i}]_{k_i-1}\},\\
    \label{eq:line2}
    \pi_1(\overline{\mathcal M}_{\mathbf k,\mathbf s}) & = \{P\in \RR[x,y]_{k-1}: P=P_L\cdots P_1,\quad  P_i\in \mathbb{C}[x^{S_i},y^{S_i}]_{k_i-1}\},\\
    \label{eq:line3}
    \pi_1^{\mathbb C}\left(\mathcal M^{\mathbb{C}}_{\mathbf k,\mathbf{s}}\right)=\pi_1^{\mathbb C}(\overline{\mathcal M}^{\mathbb{C}}_{\mathbf k,\mathbf s})&=\{P\in \mathbb{C}[x,y]_{k-1}:P=P_L\cdots P_1, \quad P_i\in \mathbb{C}[x^{S_i},y^{S_i}]_{k_i-1} \}.
\end{align}
\end{proposition}

Proposition \ref{prop:functionspace-poly} implies that $\mathcal{M}^{\mathbb{C}}_{\mathbf{k},\mathbf{s}}$ is the complexification of $\overline{\mathcal{M}}_{\mathbf k,\mathbf s}$. Thus, the term ``complex neurovariety" aptly refers to its Zariski closed characteristics.
\paragraph{Observation:}If a stride $s_i$ is equal to $1$ for some $1 \le i \le L-1$ then we can consider another architecture $(\Tilde{\mathbf{k}},\Tilde{\mathbf{s}})$ that is obtained by merging the $(i+1)$-th and $i$-th layers in $(\mathbf{k},\mathbf{s})$.
By Proposition \ref{prop:functionspace-poly}, it is clear that both neurovarities 
$\overline{\cM}_{\mathbf{k},\mathbf{s}}$ and $\overline{\cM}_{\Tilde{\mathbf{k}},\Tilde{\mathbf{s}}}$ are equal. 
For instance, consider the architecture $(\mathbf{k},\mathbf{s}) = ((2,2,2),(1,2,1))$. Then, we have 
     \begin{align}    \pi_1(\overline{\cM}_{\mathbf{k},\mathbf{s}})&=\{(ex^2+fy^2)(cx+dy)(ax+by)\in \RR[x,y]_4: a,b,c,d,e,f \in \CC\}\\
     \label{eq:reduced-k=3,2}
     &=\{(dx^2+ey^2)(ax^2+bxy+cy^2) \in \RR[x,y]_4: a,b,c,d,e \in \CC\}.
     \end{align}
     The set \eqref{eq:reduced-k=3,2} describes the polynomials in $\pi_1(\overline{\mathcal{M}}_{\Tilde{k},\Tilde{s}})$, where $(\Tilde{\mathbf{k}},\Tilde{\mathbf{s}})=((3,2),(2,1))$. This observation leads us to the following definition:
     \begin{definition}
         \label{def:reduced}
         An architecture $(\mathbf{k},\mathbf{s})$ is \emph{reduced} if $k_i>1$ for every $i$ and $s_i>1$ for $1\le i \le L-1$.
     \end{definition}
\subsection{Segre Varieties and Neurovarieties}
\label{sec:segre}
Given a tuple $\mathbf{k}=(k_1,\ldots,k_L)\in \NN^L$, the Segre variety $\cS^{\CC}_{\mathbf{k}}\subset \CC^{k_1\times \cdots \times k_L}$ is the space of rank one tensors
\begin{equation}
    \label{eq:ten-ten}
    \cS^{\CC}_{\mathbf{k}}= \{w_1 \otimes \cdots \otimes w_L \mid w_i \in \CC^{k_i}, 1\le i\le L\}.
\end{equation}
The Segre variety $\cS^{\CC}_{\mathbf{k}}$ can be obtained by the common zero locus of all $2 \times 2$ minors derived from a symbolic tensor of size $\mathbf{k}$. We also denote the real rank one tensor as $\cS_{\mathbf{k}}$.

The connection between Segre varieties $\cS^{\CC}_{\mathbf{k}}$ and neurovarieties ${\mathcal{M}}_{\mathbf{k},\mathbf{s}}^{\CC}$ is motivated by Example \ref{ex:L2-s_1=2}. In fact, $\cM_{\mathbf{k},\mathbf{s}}^{\CC}$ is equal to $\Psi(\cS^{\CC}_{\mathbf{k}})$ where $\Psi:\CC^{k_1\times \cdots \times k_L}\to \CC^k$ is a linear map. For instance, $\mathcal{M}^{\CC}_{(2,2),(1,1)}$ is the image of the following linear map:
\begin{align*}
    \Psi|_{\cS^{\CC}_{(2,2)}}:\cS^{\CC}_{(2,2)} &\longrightarrow \mathcal{M}^{\CC}_{(2,2),(1,1)}, \\
    (w_2[0],w_2[1]) \otimes (w_1[0],w_1[1]) &\mapsto (w_2[0]w_1[0],w_2[0]w_1[1]+w_2[1]w_1[0],w_2[1]w_2[1]).
\end{align*}
The linear map $\Psi$ induces a rational map $\PP({\Psi})$ from  $\PP^{k_1\times \cdots \times k_L-1}$ to $\PP^{k-1}$. Since both $\PP(\cS^{\CC}_{\mathbf{k}})$ and $\PP(\cM^{\CC}_{\mathbf{k},\mathbf{s}})$ are parametrized by $\PP^{k_1-1} \times \cdots \times \PP^{k_L-1}$, the restriction $\PP(\Psi)|_{\PP(\cS_{\mathbf{k}}^{\CC})}$ is a morphism. Moreover, this morphism is birational if the architecture $(\mathbf{k},\mathbf{s})$ is reduced and $L>1$ (\cite[Corollary~5.4]{kohn2023function}).

\subsubsection{Polar Classes of Segre Varieties and Neurovarieties}
Consider a projective variety $\cX \subset \mathbb{P}^{n-1}$, its dual variety $\mathcal{Y} \subset \mathbb{P}^{n-1}$, and the $(n-2)$-dimensional conormal variety $\mathcal{N}_{\cX,\mathcal{Y}} \subset \mathbb{P}^{n-1} \times \mathbb{P}^{n-1}$. The $i$-th \emph{polar class} of $\cX$, denoted as $\delta_i(\cX)$, counts the points in $\mathcal{N}_{\cX,\mathcal{Y}} \cap (L_1 \times L_2)$, where $L_1$ and $L_2$ are generic linear spaces of $\mathbb{P}^{n-1}$ with dimensions $n-i-1$ and $i+1$, respectively. 

\begin{lemma}
    \label{lem:segre+neuro}
    Let $(\mathbf{k},\mathbf{s})$ be a reduced architecture with $L > 1$. Then, the polar classes associated with the Segre variety $\cS_{\mathbf{k}}^{\CC}$ and the neurovariety $\cM^{\CC}_{\mathbf{k},\mathbf{s}}$ are equal.
\end{lemma}
\begin{proof}
The polar classes are known to be invariant under proper generic projections (\cite{piene1978polar} and \cite{holme1988geometric}). In our case, we aim to show that the image of the projection $\PP(\Phi \circ \Psi): \PP(\cS_{\mathbf{k}}^{\CC}) \to \PP^{n-1}$ preserves the polar classes. Here, $\PP(\Phi): \PP(\cM^{\CC}_{\mathbf{k},\mathbf{s}}) \to \PP^{n-1}$ represents a generic projection, and $n$ satisfies the relation $1+\dim \PP(\cM^{\CC}_{\mathbf{k},\mathbf{s}})=1+ \sum_{i=1}^{L} (k_i-1)< n \le k$. Note that, since $\PP(\Psi)|_{\PP(\cS_{\mathbf{k}}^{\CC})}$ is a morphism, the center of the projection $\PP(\Psi)$ is disjoint from $\PP(\cS_{\mathbf{k}}^{\CC})$. Additionally, given that $(\mathbf{k},\mathbf{s})$ is reduced and $L>1$, we have $1+\dim \PP(\cM^{\CC}_{\mathbf{k},\mathbf{s}})<k$. Under these conditions, it follows that for every $i$, $\delta_i(\PP(\cS_{\mathbf{k}}^{\CC}))$ and $\delta_i(\PP(\cM^{\CC}_{\mathbf{k},\mathbf{s}}))$ are equal (see the proof of \cite[Theorem~4.1]{piene1978polar}).
\end{proof}
\subsection{(Generic) Euclidean distance degree}
Let $\mathcal{V}$ be a real algebraic variety in $\mathbb{R}^n$, and denote its complexification in $\mathbb{C}^n$ as $\mathcal{V}^{\mathbb{C}}$. The respective smooth locus is denoted by $\mathcal{V}_{\operatorname{reg}}$ and $\mathcal{V}_{\operatorname{reg}}^{\mathbb{C}}$.

For given data $u \in \RR^n$ and a symmetric matrix $T \in \RR^{n \times n}$, we define the \emph{squared distance function} $\operatorname{dist}_{T,u}$ as $\operatorname{dist}_{T,u}^{\CC} |_{\mathcal{V}_{\operatorname{reg}}}$, where $\operatorname{dist}_{T,u}^{\CC}$ is given by
\begin{align}
    \label{eq:dist}
    \operatorname{dist}_{T,u}^{\CC}: \mathcal{V}_{\operatorname{reg}}^{\CC} \to \CC, \quad x \mapsto (x - u)^\top \cdot T \cdot (x - u).
\end{align}
\begin{definition}
\label{def:ed-ged}
    Let $u \in \RR^n$ be generic data, and let $T \in \RR^{n \times n}$ be a generic symmetric matrix (resp., $T=I_n$). We denote the \emph{generic EDdegree} (resp., \emph{EDdegree}) of $\mathcal{V}$ as the number of complex critical points of \eqref{eq:dist}. We write this degree as $\operatorname{EDdeg}(\cV)$ (resp., $\operatorname{gEDdeg}(\cV)$).
\end{definition}
We observe that if the variety $\mathcal{V}$ is in general coordinates, then $\operatorname{gEDdeg}(\mathcal{V}) = \operatorname{EDdeg}(\mathcal{V})$. In fact, when $u \in \mathbb{R}^n$ is generic data, $\operatorname{gEDdeg}(\mathcal{V})$ represents the highest number of critical points for any quadratic distance function $\operatorname{dist}_{T,u}^\mathbb{C}$ (\cite{maxim2020defect}).
\begin{remark}
    \label{rmk:polar}
    We can extend the Definition \ref{def:ed-ged} for projective varieties. If $\cV \subset \PP^{n-1}$ then we define $\operatorname{EDdeg(\cV)}:= \operatorname{EDdeg(C(\cV))}$ and $\operatorname{gEDdeg(\cV)}:= \operatorname{gEDdeg(C(\cV))}$, where $C(\cV)$ is the affine cone of $\cV$. It is known that the generic EDdegree of a projective variety $\cV\subset \PP^{n-1}$ is equal to the sum of its polar classes $\sum_{i=0}^{n-2} \delta_i(\cV)$ (\cite[Corollary~6.1]{draisma2016euclidean}). Consequently, by Lemma \ref{lem:segre+neuro}, if the architecture $(\mathbf{k},\mathbf{s})$ is reduced and $L>1$ then we deduce that $\operatorname{gEDdeg}(\cS_{\mathbf{k}}) = \operatorname{gEDdeg}(\overline{\cM}_{\mathbf{k},\mathbf{s}})$.
\end{remark}
\begin{example}
\label{ex:norms}
    We denote $\cV_{m\times n}^r$ (resp., $\cV_{m\times n}$) to be the space of rank $r$ (resp., full rank) matrices of size $m \times n$, where $r \le \min(m,n)$. For generic matrices $U\in \cV_{m \times n}, B\in \cV_{mn\times mn}$ and $A \in \cV_{n \times k}$ with $k\ge n$, we consider the following optimization problems:
    \begin{align}
    \label{eq:frob-matrix}
        &\min_{M \in \cV_{m\times n}^r}\lVert M-U \rVert_F^2  = \min_{M \in \cV_{m\times n}^r} \operatorname{tr}((M-U) (M-U)^\top)\\
        \label{eq:weighted-frob}
        &\min_{M \in \cV_{m\times n}^r}\lVert M-U \rVert_{AA^\top}^2 =\min_{M \in \cV_{m\times n}^r} \operatorname{tr}((M-U) AA^\top (M-U)^\top)\\
        \label{eq:gen-frob}
        &\min_{M \in \cV_{m\times n}^r} 
        \operatorname{vec}(M-U)^\top \cdot B \cdot \operatorname{vec}(M-U)=\min_{M \in \cV_{m\times n}^r} 
        \operatorname{vec}(M-U)^\top \cdot \frac{B+B^\top}{2} \cdot \operatorname{vec}(M-U),
    \end{align}
    where $\operatorname{vec}(M-U)\in \RR^{mn}$ is the vectorization of the matrix $M-U$.
    By the Eckart--Young Theorem, one can verify that \eqref{eq:frob-matrix} has $\binom{\min(m,n)}{r}$ critical points, all belonging to the smooth locus of $\cV_{m\times n}^r$. This number represents the EDdegree of $\cV_{m\times n}^r$. The complex critical points associated with \eqref{eq:gen-frob} are enumerated by the generic EDdegree of $\cV_{m\times n}^r$ (\cite[Example~7.11]{draisma2016euclidean} and \cite[Section~A.2]{geometryLinearNets}). Finally, the optimization problem \eqref{eq:weighted-frob} can be written in the form of \eqref{eq:frob-matrix}. To see this, note that $M':= M(AA^\top)^{\frac{1}{2}}$ belongs to $\cV_{m \times n}^r$ and $U':= U(AA^\top)^{\frac{1}{2}}$ is a generic matrix in $\cV_{m \times n}$. Hence, both \eqref{eq:weighted-frob} and \eqref{eq:frob-matrix} admit the same number of critical points, which is the EDdegree of $\cV_{m\times n}^r$.
\end{example}

\begin{remark}
    In Example \ref{ex:norms}, we observed that on the variety $\cV_{m\times n}^r$, both optimization problems with Frobenius norm $\lVert \cdot \rVert_{F}$
    and weighted Frobenius norm $\lVert \cdot \rVert_{AA^\top}$ admit the same number of critical points. However, in Section \ref{sec:Eucl}, we will see that this parity does not hold universally for all subvarieties $\cV \subset \cV_{m\times n}$. 
\end{remark}

\begin{table}[H]
    \centering
    \begin{tabular}{l p{11cm}}
        \textbf{Notation} & \textbf{Description} \\
        \midrule
        $\mathbf{k}=(k_1,\ldots,k_L)$ & Tuple of filter sizes\\
        $\mathbf{s}=(s_1,\ldots,s_{L-1},1)$ & Tuple of strides\\
        $(\mathbf{k}, \mathbf{s})$ and $(\tilde{\mathbf{k}}, \tilde{\mathbf{s}})$ & 1D-LCN architecture and its reduced counterpart\\
        $S_\ell$ & Shorthand for $\prod_{i=1}^{\ell-1}s_i$\\
        $k$ & Size of the output filter with value $k=k_1+\sum_{l=2}^L(k_l-1)S_l$\\
        $\mathcal{M}_{\mathbf{k}, \mathbf{s}}$ & Neuromanifold of a 1D-LCN, as a subset of $\mathbb{R}^k$\\
        $\overline{\mathcal{M}}_{\mathbf{k}, \mathbf{s}}$ & Neurovariety of 1D-LCN obtained by Zariski closure of $\mathcal{M}_{\mathbf{k}, \mathbf{s}}$\\
        $\pi_s$ & Polynomial coefficient map; see~\eqref{eq:polynomials}\\
    \end{tabular}
    \caption{List of notations used in 1D-LCNs.}
    \label{tab:symbols}
\end{table}

\section{The ideal of the complex neurovariety}
\label{subsec:ideal_functionspace}
In this section, we present an algorithm designed to generate a finite set of homogeneous polynomials with integer coefficients such that the zero locus of these polynomials over real and complex numbers corresponds to $\overline{{\mathcal M}}_{\mathbf k,\mathbf s}$ and ${\mathcal M}^{\mathbb{C}}_{\mathbf k,\mathbf s}$, respectively.  To streamline the discussion and leverage the advantages of the algebraically closed field $\CC$, our focus in this section is solely on complex neurovarieties.
\subsection{Decomposition of Sparse Polynomials}
\begin{definition}
\label{decompose-univariate}
Let $s$ be a positive integer.
A univariate polynomial $P \in \mathbb{C}[x]$ can be uniquely written as
$P(x) = p_1(x^{s}) + xp_2(x^{s})+\cdots + x^{s-2}p_{s-1}(x^{s})+x^{s-1}p_{s}(x^{s})$.
We refer to the polynomials $p_i(x^s) \in \mathbb{C}[x^s]$ as the \emph{$s$-decomposition} of $P$. 
\end{definition}
We can extend the $s$-decomposition to homogeneous polynomials $P\in \mathbb{C}[x,y]_m$.
For a given integer $a$, we denote by $\overline{a}^{s}$, the unique reminder in the integer division of $a$ by $s$.
Homogenizing the $s$-decomposition yields
\begin{equation}  
\label{eq:decompose}
P = 
y^{\overline{m}^s}p_1 + x y^{\overline{m-1}^{s}}p_2+ \cdots 
+x^{s-2}y^{\overline{m-s+2}^{s}}p_{s-1}+x^{s-1}y^{\overline{m-s+1}^{s}}p_{s},
\end{equation}
where $p_i \in \mathbb{C}[x^s,y^s]_{\lfloor \frac{m+1-i}{s} \rfloor}$.
     This decomposition defines an isomorphism $$\sigma_s:\mathbb{C}[x,y]_{m}\to \prod_{i=1}^{s} \mathbb{C}[x^s,y^s]_{\lfloor \frac{m+1-i}{s} \rfloor}, \quad P \to (p_1,\ldots,p_s),$$
     and we denote by $\CC[x^s,y^s]_{<0}$ the zero ring.
When applying the map $\sigma_{s_1}$ to a polynomial $P = P_L \cdots P_1$ that is factorized according to \eqref{eq:line3}, we observe that 
\begin{equation}
    \label{eq:s-decompLCN}
    \sigma_{s_1}(P) = \sigma_{s_1}(P_L \cdots P_1) = P_L \cdots P_2 \cdot\sigma_{s_1}(P_1)=P_L\cdots P_3 (P_2\cdot p_{1,1},\ldots,P_2\cdot p_{1,s_1}).
\end{equation}
In \eqref{eq:s-decompLCN}, note that the polynomials $P_2$ and $p_{1,i}$ both belong to $\CC[x^{s_1},y^{s_1}]$, allowing us to consider the multiplication $P_2\cdot p_{1,i}$ as a polynomial associated with a single layer. Consequently, the $s_1$-decomposition of $\cM^{\CC}_{\mathbf{k},\mathbf{s}}$ with $L$ layers leads to $s_1$ complex neurovarieties, each with $L-1$ layers. This fact facilitates the use of induction for computing the ideal of $\cM^{\CC}_{\mathbf{k},\mathbf{s}}$.
\begin{proposition}
\label{prop:decomp_poly_general_star}
Let $\mathbf{k} = (k_1,\ldots,k_L)$ and $\mathbf{s} = (s_1,\ldots,s_L)$ with $L\ge 2$. Then, we have     
\begin{equation}
\label{eq:functionspace_decomp_star}
\pi_1^{\mathbb{C}}\left(\mathcal{M}^{\mathbb{C}}_{\mathbf{k},\mathbf{s}}\right) = \left\lbrace Q \in \pi_1^{\mathbb{C}}\left(\mathcal{M}^{\mathbb{C}}_{\mathbf{k}',\mathbf{s}'}\right)   \; : \; \begin{array}{l}
        
         \gcd(\sigma_{s_1}(Q)) \in \mathbb{C}[x^{s_1},y^{s_1}]_{\ge \frac{k-k_1}{s_1}}
       \end{array}\right\rbrace, 
\end{equation}
where   $\mathbf k' = (k_1+s_1(k_2-1),k_3,\ldots,k_L)$ and $\mathbf s'= (s_1s_2,s_3,\ldots,s_L)$.
\end{proposition}
\begin{proof}
By Proposition \ref{prop:functionspace-poly}, every polynomial $Q \in \pi_1^{\mathbb{C}}\left(\mathcal{M}^{\mathbb{C}}_{\mathbf{k}',\mathbf{s}'}\right)$ admits a factorization $Q = Q_{L-1}\cdots Q_1$ according to architecture $(\mathbf k',\mathbf s')$. 
Writing $\sigma_{s_1}(Q_1) = (q_{1,1}, \ldots, q_{1,s_1})$, we see
as in \eqref{eq:s-decompLCN} that
 $\sigma_{s_1}(Q)= (Q_{L-1}\cdots Q_2 q_{1,1},\ldots, Q_{L-1}\cdots Q_2 q_{1,s_1})$.
 Since $\deg(Q_{L-1} \cdots Q_2) = \deg(Q)-\deg(Q_1) = k-k_1-s_1(k_2-1)$,  
 the condition that $\sigma_{s_1}(Q)$ has at least $\frac{k-k_1}{s_1}$ common binomials of the form $ax^{s_1}+by^{s_1}$ implies that each $q_{1,i}$ can be written as $ Q_1 q_{0,i}$, where $Q_1\in \mathbb{C}[x^{s_1},y^{s_1}]_{k_2-1}$. Reversing the $s_1$-decomposition gives a factorization $Q = Q_{L-1}\cdots Q_2 Q_1 Q_0$
 according to the architecture $(\mathbf{k},\mathbf{s})$, where $\sigma_{s_1}(Q_0) = (q_{0,1},\cdots,q_{0,s_1})$.
 This shows that $Q \in \pi_1^{\mathbb{C}}\left(\mathcal{M}^{\mathbb{C}}_{\mathbf{k},\mathbf{s}}\right)$.
 
 Conversely, if $P \in \pi_1^{\mathbb{C}}\left(\mathcal{M}^{\mathbb{C}}_{{ \mathbf k},{\mathbf s}}\right)$, then $P$ can be written as $P_L \cdots P_1$ with $P_i \in \mathbb{C}[x^{S_i},y^{S_i}]_{k_i-1}$; see \eqref{eq:pol-mult}. Note that $P_L\cdots P_3 (P_2P_1)$ is also contained in $\pi_1^{\mathbb{C}}\left(\mathcal{M}^{\mathbb{C}}_{\mathbf k',\mathbf s'}\right)$, and by applying the $s_1$-decomposition to $P$, we get $\sigma_{s_1}(P)=(P_L\cdots P_2 p_{1,1},\ldots,P_L\cdots P_2 p_{1,s_1})$ as in \eqref{eq:s-decompLCN}, which shows the desired gcd condition. 
\end{proof}

\begin{corollary}
 \label{cor:gcd_condition_l2}
The variety $\pi_1^{\mathbb{C}}\left(\mathcal{M}^{\mathbb{C}}_{(k_1,k_2),(s_1,1)}\right)$ is the set of all polynomials $Q \in \mathbb{R}[x,y]_{k-1}$ with $\sigma_{s_1}(Q)=(q_1,\ldots,q_{s_1})$  such that
     \begin{equation}
         \label{eq:l2_gcd}
          \gcd(q_{1},\ldots,q_{s_1}) \in \mathbb{C}[x^{s_1},y^{s_1}]_{\ge k_2-1}.
     \end{equation}
 \end{corollary}
\begin{proof}
    Since $L=2$, the single-layer neurovariety $\mathcal{M}_{k',s_1}^{\CC}$ is a vector space isomorphic to $\CC^{k'}$, where $k'=k=k_1+s_1(k_2-1)$. Also, as $k_2-1 = \frac{k-k_1}{s_1}$, then
    the condition  \eqref{eq:functionspace_decomp_star} is equivalent to \eqref{eq:l2_gcd}.
\end{proof}
\paragraph{Elaboration on Example \ref{ex:L2-s_1=2}:} According to Corollary \ref{cor:gcd_condition_l2}, a polynomial $P= Ax^3+Bx^2y+Cxy^2+Dy^3$ belongs to $\pi_1^{\mathbb{C}}\left(\mathcal{M}^{\mathbb{C}}_{(2,2),(2,1)}\right)$ if and only if $q_1 = Bx^2+Dy^2$ and $q_2 = Ax^2+Cy^2$ share a common root in $\CC[x^2,y^2]$. By performing a change of variables, substituting $x^2$ with $x$ and $y^2$ with $y$, we can ascertain that the given condition holds true if and only if $AD-BC=0$.
\begin{example}
\label{ex:l2-poly}
    \sloppy
    Consider the two-layer neurovariety $\mathcal{M}^{\mathbb{C}}_{(3,2),(2,1)}$. The polynomials in $\pi_1^{\mathbb{C}}\left(\mathcal{M}^{\mathbb{C}}_{(3,2),(2,1)}\right)$ are of the form $Q=Ax^4+Bx^3y+Cx^2y^2+Dxy^3+Ey^4$ such that it admits the factorization
    \begin{equation}
        \label{eq:ex-L=2,k=3,2}
        P_2P_1=(dx^2+ey^2)(ax^2+bxy+cy^2).
    \end{equation}
    According to Corollary \ref{cor:gcd_condition_l2}, $Q$ exhibits the factorization \eqref{eq:ex-L=2,k=3,2} if and only if $q_1$ and $q_2$ have a common factor in 
    $\mathbb{C}[x^2,y^2]_{\ge 1}$, where $(q_1,q_2) = \sigma_2(Q) = (Ax^4+Cx^2y^2+Ey^4,Bx^2+Dy^2)$. 
\end{example}

\subsection{Polynomial Resultants and \texorpdfstring{$\mathcal{M}^{\mathbb{C}}_{\mathbf{k},\mathbf{s}}$}{M^C\_k,s}}
\label{sec:equa} 
The condition in \eqref{eq:l2_gcd} had been studied intensively in the past years e.g., 
in \cite{kakie1976resultant}. 
 To provide an algorithm that determines polynomial equations cutting out $\mathcal{M}^{\mathbb{C}}_{\mathbf k,\mathbf s}$,
let us consider $s$ homogeneous polynomials $q_{i} \in \mathbb{C}[x,y]_{n_i \ge 1} \setminus \{0\}$ that are written as
\begin{equation}
    \label{eq:polynomia-result-kak}
    q_{i}(x,y)= A_{0}^{(i)}x^{n_{i}}+A_{1}^{(i)}x^{n_{i}-1}y\cdots +A_{n_{i}}^{(i)}y^{n_{i}}.
\end{equation}
Let $n_{*} := \min \{n_{i}: 1 \le i \le s\}$, and $n^{*} := \max \{n_{i}: 1 \le i \le s\}$. For every $l\ge 0$, we define the following resultant matrix:
$$R_l:=\begin{bNiceMatrix}
   A_0^{(1)}& A_1^{(1)} &\Cdots & A_{n_1}^{(1)} & & &\\
   &A_0^{(1)}& A_1^{(1)} &\Cdots & A_{n_1}^{(1)} & & \\
   & &  & \cdots & \Cdots & \cdots & \\ 
   & & & A_0^{(1)}& A_1^{(1)} &\Cdots & A_{n_1}^{(1)} \\
   \Cdots\\
   A_0^{(s)}& A_1^{(s)} &\Cdots & A_{n_s}^{(s)} & & &\\
   &A_0^{(s)}& A_1^{(s)} &\Cdots & A_{n_s}^{(s)} & & \\
   & &  & \cdots & \Cdots & \cdots & \\ 
   & & & A_0^{(s)}& A_1^{(s)} &\Cdots & A_{n_s}^{(s)}
\end{bNiceMatrix},$$
where there are $\min(0,l-n_{i}+1)$ rows associated with $q_{i}$ for every $i$.
\begin{theorem}[\cite{kakie1976resultant}, Theorem~A.]
\label{thm:kakie}
 The homogeneous polynomials $q_1,\ldots,q_s$ defined in \eqref{eq:polynomia-result-kak} have a common factor of degree at least $m$ if and only if the rank of the resultant matrix $R_{n^{*}+n_{*}-m}$ is less than $n_* +n^* -2m + 2$.    
\end{theorem}

\begin{remark}
\label{rmk:alphastar}
The homogeneous polynomial $q_i$ in Corollary \ref{cor:gcd_condition_l2} belongs to $\CC[x^{s_1},y^{s_1}]_{\lfloor \frac{k-i}{s_1} \rfloor}$. Therefore, there exists a positive integer $r$ such that the first $r$ polynomials $q_1,\ldots, q_{r}$ are in $\CC[x^{s_1},y^{s_1}]_{n^*}$, where $n^*=\lfloor \frac{k-1}{s_1} \rfloor$ and the remaining polynomials $q_{r+1},\ldots,q_{\min(k_1,s_1)}$ are in $\CC[x^{s_1},y^{s_1}]_{n^*-1}$. If $s_1>k_1$ then $q_{k_1+1},\ldots,q_{s_1}$ are zero polynomials.
\end{remark}

\begin{corollary}
     \label{cor:result-two-layer}
     Let $q_1,\ldots,q_{s_1}$ be the homogeneous polynomials in Corollary \ref{cor:gcd_condition_l2} after the change of variables $x^{s_1} \to x$ and $y^{s_1}\to y$. Then,  $\pi_1^{\CC}\left(\cM^{\CC}_{(k_1,k_2),(s_1,1)}\right)$ is the zero locus of the ideal $\mathfrak{I}=\mathfrak{I}_1+\mathfrak{I}_2$ such that
     \begin{enumerate}
         \item $\mathfrak{I}_1$ is the ideal generated by all minors of size $2n^* -2k_2+3$ in the resultant matrix $R_{l_1}$ of $q_1,\ldots,q_{s_1}$, where $l_1=2n^*-k_2$, and
         \item $\mathfrak{I}_2$ is the ideal generated by all minors of size $2n^*-2k_2+4$ in the resultant matrix $R_{l_2}$ of $q_1,\ldots, q_r$, where $l_2 = 2n^*-k_2+1$,
     \end{enumerate}
     and $n^*, r$ are computed in Remark \ref{rmk:alphastar}. In particular, if $r=1$ or $\min(k_1,s_1)=r$ then the ideal $\mathfrak{I}_2$ is trivial or equal to $\mathfrak{I}_1$ respectively, hence $\mathfrak{I} = \mathfrak{I}_1$.
\end{corollary}
\begin{proof}
    In this setting, the gcd condition in Corollary \ref{cor:gcd_condition_l2} translates to $\gcd(q_1,\ldots,q_{s_1})\in \CC[x,y]_{m\ge k_2-1}$. First, according to Theorem \ref{thm:kakie}, the rank of the resultant matrix $R_{l_1}$ of polynomials $q_1,\ldots, q_{s_1}$ must be less than $n^*+(n^*-1)-2k_2+4$, where $l_1 = n^*+(n^*-1)-(k_2-1)$. This defines the procedure for computing the ideal $\mathfrak{I}_1$. Now, if $q_{r+1}=\cdots= q_{\min(k_1,s_1)}=0$, where $r\ge 2$ and $\min(k_1,s_1)-r\ge 1$, we require another resultant matrix $R_{l_2}$ to satisfy the gcd condition. In this case, the procedure for computing the ideal $\mathfrak{I}_2$ is determined by the condition that the rank of $R_{l_2}$ is less than $2n^*-2k_2+4$, where $l_2=2n^*-(k_2-1)$.
    The last case to consider is when $q_1=\cdots q_{r}=0$. We argue that this is already covered by the ideal $\mathfrak{I}_1$. To see this, note that there exists a submatrix in $R_{l_1}$ such that the rank condition provides
    \begin{equation}
        \label{eq:result-q_i}
        \gcd(q_{r+1}, \ldots,x\cdot q_{r+i}, \ldots,q_{\min(k_1,s_1)})\in \CC[x,y]_{m\ge k_2-1},\quad  1\le i\le \min(k_1,s_1).
    \end{equation}
    It is apparent that in this case, \eqref{eq:result-q_i} implies the gcd condition in \eqref{cor:gcd_condition_l2}.
\end{proof}

\paragraph{Elaboration on Example \ref{ex:l2-poly}:} Let $q_1 = Ax^2 + Cxy + Ey^2$ and $q_2 = Bx + Dy$, yielding $n^*=2$, $n_*=1$, $m=k_2-1=1$, and $r=1$. According to Corollary $\ref{cor:result-two-layer}$, the ideal $\mathfrak{I}$ can be computed only by ensuring $\rank R_{l_1}=\rank R_2 < 3$. The latter is equivalent to satisfying the following relation
\[\begin{vmatrix}
    A & C & E\\
    B & D & 0\\
    0 & B & D
\end{vmatrix} = AD^2+ B^2E-BCD=0.\]

The next example demonstrates that computing $\mathfrak{I}_2$ is necessary. In fact when $L=2$, and $s_1>2$, then the ideal $\mathfrak{I}_2$ might not be trivial.
\begin{example}
    Consider the architecture $\mathbf{k}=(5,2)$ and $\mathbf{s}=(3,1)$. A polynomial $P$
    $$P = Ax^7+Bx^6y+Cx^5y^2+Dx^4y^3+Ex^3y^4+Fx^2y^5+Gxy^6+Hy^7 \in \CC[x,y]_7$$
    is in $\pi_1^{\mathbb{C}}\left(\mathcal{M}^{\mathbb{C}}_{\mathbf{k},\mathbf{s}}\right)$ if and only if it admits the following factorization:
    $$P_2P_1=(fx^3+gy^3)(ax^4+bx^3y+cx^2y^2+dxy^3+ey^4).$$
    To find the ideal $\mathfrak{I}$, first, by Corollary \ref{cor:gcd_condition_l2}, we consider symbolic homogeneous polynomials  $$q_1=Bx^2+Exy+Hy^2, \quad q_2=Ax^2+Dxy+Gy^2, \quad q_3=Cx+Fy$$ such that the condition $\gcd(q_1,q_2,q_3)\in \CC[x,y]_{m\ge 1}$ gives the desired neurovariety. We compute $n^*=2$, $m=k_2-1=1$, and $r=2$. Then the ideal $\mathfrak{I}_1$ (resp., $\mathfrak{I}_2$) is generated by all $3\times 3$ (resp., $4\times 4$) minors in $R_{l_1}=R_2$ (resp., $R_{l_2}=R_3$). The vanishing locus of ideal $\mathfrak{I}=\mathfrak{I}_1+\mathfrak{I}_2$  is $\pi_1^{\mathbb{C}}\left(\mathcal{M}^{\mathbb{C}}_{\mathbf{k},\mathbf{s}}\right)$, which has dimension $6=k_1+k_2-1$; see Section \ref{sec:neuroman}. According to \texttt{Macaulay2}, $\sqrt{\mathfrak I}$ is a prime ideal over the rational numbers and minimally generated by 
    \begin{align*}
   \sqrt{\mathfrak{I}}= \langle &CEG-BFG-CDH+AFH, CEF-BF^2-C^2H, \\
    &CDF-AF^2-C^2G, BDF-AEF-BCG+ACH, \\
    &BDEG-AE^2G-B^2G^2-BD^2H+ADEH+2ABGH-A^2H^2 \rangle,
\end{align*}
 while $\sqrt{\mathfrak{I}_1}$ is not a prime ideal and has a primary decomposition $\sqrt{\mathfrak{I}} \cap \langle F,C \rangle \neq \sqrt{\mathfrak I}$.
\end{example}
In the next step, we introduce an algorithm that generates the polynomial equations defining every
$\mathcal{M}^{\mathbb{C}}_{\mathbf{k},\mathbf{s}}$, recursively. 

\begin{algorithm}[H]
\caption{Vanishing ideal of $\mathcal{M}^{\mathbb{C}}_{\mathbf{k},\mathbf{s}}$ using recursion}
\label{factorial}
\hspace*{\algorithmicindent} \textbf{Input:} $(L,\mathbf{k},\mathbf{s})$\\
    \hspace*{\algorithmicindent} \textbf{Output:} vanishing ideal of $\mathcal{M}^{\mathbb{C}}_{\mathbf{k},\mathbf{s}}$

\begin{algorithmic}[1]

\Function{vanish}{$L,\mathbf{k},\mathbf{s}$}
    \If{$L = 1$}
        \State \textbf{return} $0$ 
    
    \ElsIf{$L = 2$}
        
        \State \textbf{return} $\sqrt{\mathfrak{I}}$ \Comment{See Corollary \ref{cor:result-two-layer}}
        
    \ElsIf{$L > 2$}
    \State $\mathfrak{a}:= \Call{vanish}{L-1,\mathbf{k}',\mathbf{s}'}$ \Comment{$(\mathbf{k}',\mathbf{s}')$ defined in Proposition \ref{prop:decomp_poly_general_star}}
    \State $\mathfrak{b}:= \Call{vanish}{2,(k_1, \frac{k-k_1}{s_1}+1),(s_1,1)}$
        \State \textbf{return} $\sqrt{\mathfrak{a}+\mathfrak{b}}$
    
    \EndIf
\EndFunction
\end{algorithmic}
\end{algorithm}
\begin{theorem}
    \label{thm:alg}
    Algorithm \ref{factorial}, correctly computes the vanishing ideal of the neurovariety $\cM_{\mathbf{k},\mathbf{s}}^\CC$.
\end{theorem}
\begin{proof}
Regarding the case where $L=1$, observe that the corresponding neurovariety is the vector space $\mathbb{R}^{k}$. Consequently, the corresponding vanishing ideal is trivially the zero ideal. The case where $L=2$ is explained in Corollary \ref{cor:result-two-layer}. For $L>2$, the equality of sets given in \eqref{eq:functionspace_decomp_star} implies that $\pi_1^{\mathbb{C}}\left(\mathcal{M}^{\mathbb{C}}_{\mathbf{k},\mathbf{s}}\right)= \pi_1^{\mathbb{C}}\left(\mathcal{M}^{\mathbb{C}}_{\mathbf{k}',\mathbf{s}'}\right)\cap \pi_1^{\mathbb{C}}\left({\mathcal{M}}^{\mathbb{C}}_{\left(k_1,\frac{k-k_1}{s_1}+1\right),(s_1,1)}\right)$. Hence the vanishing ideal of $\mathcal{M}^{\mathbb{C}}_{\mathbf{k},\mathbf{s}}$ is equal to $\sqrt{\mathfrak{a} + \mathfrak{b}}$, where $\mathfrak a$ and $\mathfrak b$ are the vanishing ideals of $\mathcal{M}^{\mathbb{C}}_{\mathbf{k}',\mathbf{s}'}$ and ${\mathcal{M}}^{\mathbb{C}}_{\left(k_1,\frac{k-k_1}{s_1}+1\right),(s_1,1)}$ respectively.
\end{proof}
\begin{remark}
    In Algorithm \ref{factorial}, we observe that the vanishing ideal of $\mathcal{M}_{\mathbf{k},\mathbf{s}}^\CC$ can be constructed from the vanishing ideals of two-layer neurovarieties with specific architectures. It is worth mentioning that in many two-layer neurovarieties, the ideal generated by the minors stated in Corollary \ref{cor:result-two-layer} is not radical. Consequently, in Algorithm \ref{factorial}, taking the radical is unavoidable to find the vanishing ideal of $\cM_{\mathbf{k},\mathbf{s}}^\CC$.
However, finding the radical of an ideal is computationally expensive, so in an alternative variant of Algorithm \ref{factorial}, we can replace $\sqrt{\mathfrak{a}+\mathfrak{b}}$ and $\sqrt{\mathfrak I}$ with $\mathfrak{a}+\mathfrak{b}$ and $\mathfrak I$, respectively. This modification results in the final ideal being crafted from polynomials with integer coefficients. The common zero locus of these polynomials is the neurovariety ${\cM}_{\mathbf{k},\mathbf{s}}^\CC$. In Table \ref{tab:time-comparison}, we provide an overview of the running times associated with our algorithm across several architectures.
\end{remark}
\begin{table}[H]
    \centering
    \begin{tabular}{c|ccccc}
         & $(k_1,k_2)=$  & $(5,5)$ & $(6,6)$ & $(7,7)$ & $(8,8)$ \\
         \hline
        Kernel   & & $0.315$ s & $8.407$ s & $816.198$ s & $-$ \\
        Ideal $\mathfrak{I}$  & & \textbf{0.008} s & \textbf{0.177} s & \textbf{26.671} s & \textbf{40.647} s \\
    \end{tabular}
    \caption{Comparison of computation times (in seconds) for finding an ideal with vanishing set $\mathcal{M}_{(k_1,k_2),(3,1)}$ using \texttt{Macaulay2} on a system equipped with an AMD Ryzen 9 5900HX processor and 16GB of RAM. Row 1 displays the running times for computing the kernel of a ring homomorphism associated with the parametrization of the neurovariety. In Row 2, the running times correspond to the computation of the ideal $\mathfrak{I}$ following the approach outlined in Corollary \ref{cor:result-two-layer}.}
    \label{tab:time-comparison}
\end{table}
\begin{example}
\label{ex:3-layer}
Consider the family of degree $k-1=8$ homogeneous polynomials
\begin{equation}
    \label{eq:degree8}
    Ax^8+Bx^7y+Cx^6y^2+Dx^5y^3+Ex^4y^4+Fx^3y^5+Gx^2y^6+Hxy^7+Iy^8 \in \CC[x,y]_8.
\end{equation}
We seek algebraic conditions ensuring \eqref{eq:degree8} admits the factorization
\begin{equation}
    \label{eq:p3.ex}
    P_3P_2P_1=(fx^4+gy^4)(dx^2+ey^2)(ax^2+bxy+cy^2) \in \mathbb{C}[x,y]_8.
\end{equation}
The polynomials in \eqref{eq:p3.ex} correspond to the complex neurovariety $\mathcal{M}^{\mathbb{C}}_{(3,2,2),(2,2,1)}$. Following Algorithm \ref{factorial}, we derive conditions for \eqref{eq:degree8} to admit the decomposition
\begin{equation}
    \label{eq:p3.ex.merge12}
    (f x^4 + g y^4)(\alpha_1 x^4+\alpha_2 x^3y + \alpha_3 x^2y^2 + \alpha_4 xy^3 + \alpha_5 y^4).
\end{equation}
These polynomials correspond to the complex neurovariety $\mathcal{M}^{\mathbb{C}}_{(5,2),(4,1)}$. Applying the $4$-decomposition $\sigma_4$ to \eqref{eq:degree8} yields
\begin{equation}
    \label{eq:p3.ex.sigma}
     (fx^4+gy^4)(\alpha_1 x^4+\alpha_5 y^4,\alpha_4,\alpha_3,\alpha_2).
\end{equation}
After the change of variables $x^4 \to x$ and $y^4 \to y$, the polynomials in \eqref{eq:p3.ex.sigma} are written as $$(Ax^2+Exy+Iy^2, Dx+Hy, Cx+Gy, Bx+Fy).$$

By Corollary \ref{cor:result-two-layer}, the ideal $\mathfrak a$ is generated only by all $3\times 3$ minors in the resultant matrix
\begin{equation}
    R_2=\begin{bNiceMatrix}
        A & E & I \\
        B & F & 0 \\
        0 & B & F \\
        C & G & 0 \\
        0 & C & G \\
        D & H & 0 \\
        0 & D & H
    \end{bNiceMatrix},
\end{equation}
giving rise to $\binom{7}{3}=35$ homogeneous polynomials of degree $3$ with integer coefficients.

To compute the ideal $\mathfrak b$, consider polynomials \eqref{eq:degree8} that admit the factorization
\begin{equation}
    \label{eq:p3.ex.merge32}
    (\beta_4 x^6+\beta_3 x^4y^2+\beta_2 x^2y^4+\beta_1 y^6)(ax^2+bxy+cy^2).
\end{equation}
These polynomials correspond to the complex neurovariety $\mathcal{M}^{\CC}_{(3,4),(2,1)}$. To find the condition for \ref{eq:p3.ex.merge32}, we apply $\sigma_2$ to \eqref{eq:degree8}, and the change of variables $x^2 \to x$, and $y^2 \to y$, we get 
\begin{equation}
    \label{eq:ex-l3-l2}(Ax^4+Cx^3y+Ex^2y^2+Gxy^3+Iy^4,Bx^3+Dx^2y+Fxy^2+Hy^3).
\end{equation}
Again by Corollary \ref{cor:result-two-layer}, the condition stipulating that the two polynomials in \eqref{eq:ex-l3-l2} share at least three common roots is precisely defined by the set of all $3\times 3$ minors in the following resultant matrix 
\begin{equation}
    \label{eq:res-5}
    R_4=
    \begin{bNiceMatrix}
        A & C & E & G & I  \\
        B & D & F & H & 0 \\
        0 & B & D & F & H
    \end{bNiceMatrix}.
\end{equation}
This gives $\binom{5}{3}=10$ homogeneous polynomials of degree $3$ with integer coefficients. Finally, the vanishing ideal associated with $\mathcal{M}^{\mathbb{C}}_{(3,2,2),(2,2,1)}$ is $\sqrt{\mathfrak a+\mathfrak b}$. According to \texttt{Macaulay2}, this ideal is prime and minimally generated by 13 polynomials as follows:
\begin{align*}
   \sqrt{\mathfrak a+\mathfrak b}= \langle &DG-CH, DF-BH, CF-BG, \\
    &FGH-EH^2-F^2 I+DHI, BGH-AH^2-BFI, \\
    &DEH-AH^2-D^2 I, CEH-AGH-CDI, \\
    &BEH-AFH-BDI, BCH-ADH-B^2 I, \\
    &CEG-AG^2-C^2 I, BEG-AFG-BCI, \\
    &BEF-AF^2-B^2 I, BCD-AD^2-B^2 E+ABF \rangle.
\end{align*}

\end{example}

\section{Optimization Complexity of 1D-LCN}
\label{sec:Eucl}
In this section, we discuss the training process of a 1D-LCN using the quadratic loss function. Our analysis aims to gauge the algebraic complexity of this training process. 
\subsection{Training 1D-LCN}
In the realm of neural network training, we are given a set of $N$ data points $X=[\mathbf{x}_1|\cdots| \mathbf{x}_N]\in \RR^{d_0\times N}$, accompanied by corresponding labels $Y = [\mathbf{y}_1|\cdots|\mathbf{y}_N]\in \RR^{d_L\times N}$. The task at hand involves training a 1D-LCN using the quadratic loss function, i.e.,
 \begin{equation}
    \label{eq:opt}
    \min_{w \in \mathcal{M}_{\mathbf{k},\mathbf{s}}} \lVert \alpha_{w,s}(X)-Y \rVert^2,
\end{equation}
where $\alpha_{w,s}$ is a convolution with filter $w$ of size $k$ and stride $s = s_1\cdots s_L$. As outlined in Section \ref{sec:neuroman}, the neuromanifold $\mathcal{M}_{\mathbf{k},\mathbf{s}}$ is a semialgebraic set. Therefore, to understand the algebraic complexity of \eqref{eq:opt}, one needs to study the geometry of $\overline{\cM}_{\mathbf{k},\mathbf{s}}$ along with the relative boundary of $\cM_{\mathbf{k},\mathbf{s}}$ within $\overline{\cM}_{\mathbf{k},\mathbf{s}}$. The latter is known to be difficult to study (\cite{kohn2023function}). However, in cases where the architecture $(\mathbf{k},\mathbf{s})$ is reduced, almost always, every critical point associated with the optimization problem $\eqref{eq:opt}$ belongs to the smooth, relative interior of $\cM_{\mathbf{k},\mathbf{s}}$ (\cite[Theorem~2.11]{kohn2023function}). Therefore, in the case of reduced architectures, we can replace \eqref{eq:opt} with the following optimization problem:
\begin{equation}
    \label{eq:optt}
    \min_{w \in \overline{\cM}_{\mathbf{k},\mathbf{s}}} \lVert \alpha_{w,s}(X)-Y \rVert^2.
\end{equation}
In the next theorem, we express the number of complex critical points of \eqref{eq:optt} in terms of the generic EDdegree of a Segre variety. For simplicity, from now on, we only consider the matrix representation of a convolution $\alpha_{w,s}:\RR^d \to \RR^{d'}$, which we refer to as the \emph{convolutional matrix}.
\begin{theorem}
    \label{thm:opt}
Consider a reduced architecture denoted by $(\mathbf{k},\mathbf{s})$ with $L>1$. For almost all data pairs $(X,Y) \in \RR^{d_0 \times N} \times \RR^{d_L \times N}$, where $N\ge d_0$, the number of complex critical points of the optimization problem \eqref{eq:optt} equals the generic EDdegree of the Segre variety $\cS_{\mathbf{k}}$.
\end{theorem}
\begin{proof}
If $X \in \RR^{d_0\times N}$ is generic with $N\ge d_0$, note that $X$ is of full rank. Hence, by defining $U:=YX^\top(XX^\top)^{-1}$, the minimization problem \eqref{eq:optt} is equivalent to 
    \begin{align}
    \label{eq:opt1}
        \min_{w \in \overline{\cM}_{\mathbf{k},\mathbf{s}}} \lVert \alpha_{w,s}-U \rVert^2_{XX^\top} &= \min_{w \in \overline{\cM}_{\mathbf{k},\mathbf{s}}} \operatorname{tr}[(\alpha_{w,s}-U)^\top XX^\top (\alpha_{w,s}-U)]\\
        \label{eq:opt2}
        &= \min_{w \in \overline{\cM}_{\mathbf{k},\mathbf{s}}} \operatorname{tr}[(\alpha_{w,s}-\alpha_{u,s})^\top XX^\top (\alpha_{w,s}-\alpha_{u,s})],
    \end{align}
where $\alpha_{u,s}$ is a convolutional matrix obtained by the projection of $U$ using the scalar product $\langle \cdot,\cdot \rangle_{XX^\top}$ onto the space of convolutional matrices with filter size $k$ and stride $s$. We are interested in expressing \eqref{eq:opt2} only with filters instead of convolutional matrices. To do that, let's denote $\operatorname{Sym}_n(\RR)$ (resp., $\operatorname{SPD}_n(\RR)$) to be the space of symmetric (resp., symmetric positive definite) matrices of size $n$ over the real numbers. We define the linear map $\psi$ from the space of square matrices of size $d_0$ to the space of square matrices of size $k$ as follows:
\begin{align}
\label{eq:psi}
 \psi:\cV_{d_0\times d_0} \to \cV_{k\times k}, \quad M  \mapsto  \left[\sum_{m=0}^{d_L-1}M_{i+sm,j+sm}\right]_{ij},
\end{align} 
where $d_0 = k+(d_L-1)s$. In other words, the map $\psi$ sums submatrices $M_0,\ldots, M_{d_L-1}$ of $M$, where each $M_r$ is
\begin{equation}
    \label{eq:M_r}
    M_r = \begin{bNiceMatrix}
    M_{1+rs,1+rs}&\cdots & M_{1+rs,k+rs}\\
    \vdots & & \vdots\\
    M_{k+rs,1+rs}& \cdots & M_{k+rs,k+rs} 
\end{bNiceMatrix}.
\end{equation}
Note that if $M$ is symmetric (resp., positive definite) then each $M_r$ is also symmetric (resp., positive definite). Hence, we can make the following useful observation:
\begin{equation}
    \label{eq:observ}
    \psi(\operatorname{Sym}_{d_0}(\RR))=\operatorname{Sym}_k(\RR)  \quad,\quad \psi(\operatorname{SPD}_{d_0}(\RR))=\operatorname{SPD}_k(\RR).
\end{equation}

Now, we can rewrite \eqref{eq:opt2} by replacing $\alpha_{w,s}$ and $\alpha_{u,s}$ with their corresponding filters as follows
\begin{equation}
    \label{eq:opt3}
    \min_{w\in \overline{\cM}_{\mathbf{k},\mathbf{s}}} (w-u)^\top\psi(XX^\top)(w-u),
\end{equation}
where according to \eqref{eq:observ}, $\psi(XX^\top)$ is a generic matrix in $\operatorname{SPD}_k(\RR)$. Since $\operatorname{SPD}_k(\RR)$ is Zariski dense in $\operatorname{Sym}_k(\RR)$, we deduce that the number of complex critical points that appear in problem \eqref{eq:opt3} is the generic EDdegree of $\overline{\cM}_{\mathbf{k},\mathbf{s}}$. According to Remark \ref{rmk:polar}, this number is equal to $\operatorname{gEDdeg}(\cS_{\mathbf{k}})$.
\end{proof}
The closed form of the generic EDdegree of $\cS_{\mathbf{k}}$ (or equivalently $\overline{\cM}_{\mathbf{k},\mathbf{s}}$), denoted by $C_{\mathbf{k}}$ is as follows:
 
\begin{equation}
\label{eq:ed-formula}
C_{\mathbf{k}}:=\sum_{t=0}^{\overline{k}} (-1)^t (2^{\overline{k}+1-t}-1)(\overline{k}-t)! \left[ \sum_{\substack{i_1+\cdots+i_L=t \\ i_j \le k_j, \forall  j}} \frac{\binom{k_1+1}{i_1}\cdots \binom{k_L+1}{i_L}}{(k_1-i_1)!\cdots (k_L-i_L)!} \right ],
\end{equation}
where $\overline{k}=\dim \cS_{\mathbf{k}} =k_1+\cdots +k_L-L$  (\cite[Proposition~3.5]{kozhasov2023minimal}). 

The following corollary directly follows from Theorem \ref{thm:opt} and the formula in \eqref{eq:ed-formula}. 
\begin{corollary}
\label{cor:arrange}
The complex critical points appearing in \eqref{eq:optt}, as per the conditions outlined in Theorem \ref{thm:opt}, demonstrate the following properties:

\begin{enumerate}
\item The count is unaffected by variations in the strides $s_i$ or the arrangement of filter sizes in the architecture.
\item It can exhibit exponential growth with respect to the filter sizes $k_i$; see  \cite[A231482]{oeis} and Table \ref{tab:c_k_twolayer}.
\end{enumerate}\end{corollary}

 \begin{table}[H]
    \centering
    \renewcommand{\arraystretch}{1.2}
    \begin{tabular}{c|ccccccccc}
        & $k_2=2$ & $k_2=3$ & $k_2=4$ & $k_2=5$ & $k_2=6$ & $k_2=7$ & $k_2=8$ & $k_2=9$  \\
        \hline
        $k_1=2$ & 6 & 10 & 14 & 18 & 22 & 26 & 30 & 34  \\
        $k_1=3$ & 10 & 39 & 83 & 143 & 219 & 311 & 419 & 543 \\
        $k_1=4$ & 14 & 83 & 284 & 676 & 1324 & 2292 & 3644 & 5444  \\
        $k_1=5$ & 18 & 143 & 676 & 2205 & 5557 & 11821 & 22341 & 38717  \\
        $k_1=6$ & 22 & 219 & 1324 & 5557 & 17730 & 46222 & 104026 & 209766  \\
        $k_1=7$ & 26 & 311 & 2292 & 11821 & 46222 & 145635 & 388327 & 910171  \\
        $k_1=8$ & 30 & 419 & 3644 & 22341 & 104026 & 388327 & 1213560 & 3288712  \\
        $k_1=9$ & 34 & 543 & 5444 & 38717 & 209766 & 910171 & 3288712 & 10218105  \\
        
    \end{tabular}
    \caption{The generic EDdegree of neurovariety $\overline{\cM}_{(k_1,k_2),(s_1,1)}$,  where $s_1>1$.}
    \label{tab:c_k_twolayer}
\end{table}
\subsection{Training with Large Number of Layers}
In the context of deep neural networks, the term ``deep" typically denotes a network with more than two layers. Numerous empirical experiments suggest that the capacity to uncover complex patterns within a given dataset is positively influenced by the increasing number of layers (\cite{uzair2020effects}). Also, while training deep models, it is yet unknown why bad local minima are not ``usually" attained by gradient methods (\cite{swirszcz2016local}). 

In the 1D-LCN setting, the number of all complex critical points in \eqref{eq:optt} is dictated by $C_{\mathbf{k}}$ as defined in \eqref{eq:ed-formula}. Many of these critical points possibly correspond to valid local minima. In such instances, one plausible explanation is that many local minima exist but are primarily concentrated in the close proximity of the global minimum or, at the very least, ``good" local minima. In deeper networks, this phenomenon becomes particularly intriguing as it appears that the number of critical points increases with the number of layers, see Figure \ref{fig:tree}.

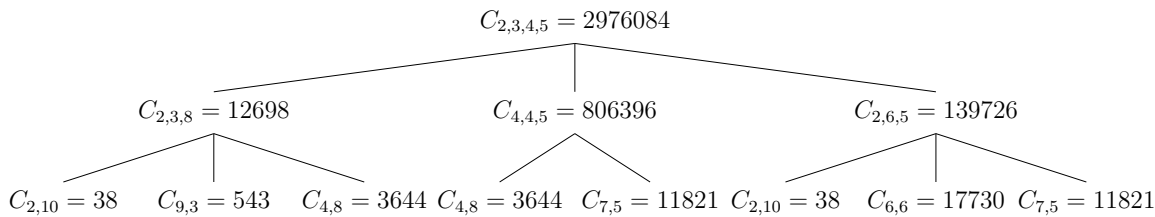
\begin{figure}[H]
\centering
\scalebox{0.8}{ 
\begin{tikzpicture}
[
    level 1/.style={sibling distance=60mm}, 
    level 2/.style={sibling distance=25mm}, 
    level 3/.style={sibling distance=20mm}, 
]
    \node {$C_{2,3,4,5}=2976084$}
    child {node {$C_{2,3,8}=12698$}
        child {node {$C_{2,10}=38$}}
        child {node {$C_{9,3}=543$}}
        child {node {$C_{4,8}=3644$}}}
    child {node {$C_{4,4,5}=806396$}
        child {node {$C_{4,8}=3644$}}
        child {node {$C_{7,5}=11821$}}    
    }
    child {
        node {$C_{2,6,5}=139726$}
        child {node {$C_{2,10}=38$}}
        child {node {$C_{6,6}=17730$}}
        child {node {$C_{7,5}=11821$}}
    };
\end{tikzpicture}
}
\caption{
The top row displays the value of $C_{2,3,4,5}$ for a reduced $4$-layer architecture. Subsequent rows showcase the computation of various $C_{\mathbf{k}}$ values for architectures derived by merging two layers at the parent node while preserving the dimension of the neuromanifold. }
\label{fig:tree}
\end{figure}

\begin{remark}
    In a fully connected linear network, the neurovariety is the space of rank $r \le \min(m,n)$ matrices $\cV_{m \times n}^r$ (\cite{geometryLinearNets}). As discussed in Example \ref{ex:norms}, training this network with the loss function \ref{eq:weighted-frob} gives rise to $\binom{\min(m,n)}{r}$ critical points, which is the EDdegree of $\cV_{m\times s}^{r}$. This number does not depend on the number of layers and is significantly lower than $C_{\mathbf{k}}$, the number of critical points that appear in training a 1D-LCN with the same number of parameters. 
\end{remark}

\subsubsection*{Acknowledgments}
I would like to express my sincere gratitude to my supervisors, Kathlén Kohn and Joakim Andén. Kathlén provided invaluable feedback, her profound insights in mathematics, and engaging conversations that greatly enriched my understanding. Joakim's feedback and applied perspectives were instrumental in shaping the direction of this work.

Special thanks to my dear friend Felix Rydell, with whom I enjoyed insightful conversations about metric algebraic geometry. Felix also provided valuable feedback on this article, contributing to its improvement.

I extend my thanks to Bernd Sturmfels and Luca Sodomaco for our discussions on Euclidean distance optimization, which added depth to my research.

\bibliographystyle{alpha}
\bibliography{literature}

\end{document}